\documentclass[a4paper,11pt]{article}
\usepackage[top=2.5cm,bottom=2.5cm,left=2.2cm,right=2.2cm]{geometry}
\usepackage{amsfonts}
\usepackage{mathrsfs,amscd,amssymb,amsthm,amsmath,bm,graphicx,psfrag,subfigure,url,mathtools}
\usepackage{pict2e}
\usepackage{stfloats}
\usepackage{psfrag,amsmath}
\usepackage{tikz}
\usepackage{indentfirst}
\usepackage[
    colorlinks=true,
    linkcolor=blue,
    citecolor=black,
    urlcolor=blue
]{hyperref}
\usepackage{bookmark}
\usepackage{enumerate}
\usepackage{latexsym,euscript,epic,eepic,color}
\usepackage{multirow}
\usepackage{multicol}
\usepackage{longtable}
\usepackage{adjustbox}
\usepackage[all]{xy}
\usepackage{setspace}
\usepackage{booktabs}
\allowdisplaybreaks
\usepackage{authblk}
\usepackage{cite}

\makeatletter

\renewcommand{\@seccntformat}[1]{{\csname the#1\endcsname}{\normalsize .}\hspace{.5em}}
\makeatother

\usepackage{ifpdf}
\usepackage{indentfirst}

\newcommand{\cP}{\mathcal{P}}
\newcommand{\binomset}[2]{\binom{#1}{#2}}
\newtheorem{thm}{Theorem}[section]
\newtheorem{theorem}[thm]{Theorem}

\newtheorem{lemma}[thm]{Lemma} 

\newtheorem{conj}{Conjecture}

\begin{document}
\baselineskip=0.23in
\title{\bf $k$-Connected Subgraphs of All Orders in Large Graphs with Minimum Degree at Least $n/q$}
\author{Heng Yang}
\affil{Department of Mathematics, East China Normal University, Shanghai, 200241, China}
\date{\today}
\maketitle

\begin{abstract}
For every fixed pair of integers $k\ge 2$ and $q\ge 3$, we prove that
every sufficiently large $k$-connected graph $G$ of order $n$ with
minimum degree $\delta(G)\ge n/q$ contains a $k$-connected subgraph of
every order $\ell\in\{2k,2k+1,\ldots,n\}$. In the case $k=2$, this confirms a
conjecture of Liu and Ning~\cite{LiuNing}. The proof combines two constructions. First, we construct a small $k$-connected subgraph $D$ such that every
vertex outside $D$ has at least $k$ neighbors in $D$. By successively adding the vertices outside $D$, we obtain
$k$-connected subgraphs of every order from $|V(D)|$ to $n$. Second, an averaging argument on common
neighborhoods, together with a complete bipartite construction, yields
$k$-connected subgraphs of every order from $2k$ to $|V(D)|$. Together,
the two constructions cover all orders from $2k$ to $n$. The lower
endpoint $2k$ is best possible.

\vskip 0.2cm
\noindent {\bf Keywords:}
Minimum degree; $k$-connected graphs; interpolation property

\vskip 0.2cm
\noindent {\bf AMS Subject Classification:}
05C40; 05C07
\end{abstract}

\section{\normalsize Introduction}

Throughout this paper, all graphs are finite, simple, and undirected. For
notation and terminology not defined here, we refer the reader to
West~\cite{West}.

Let $\mathcal{F}$ be a family of subgraphs of a graph, and let $f$ be an
integer-valued graph invariant. We say that $f$ has the
\emph{interpolation property} over $\mathcal{F}$ if, whenever
$H_1,H_2\in\mathcal{F}$ satisfy
\[
    f(H_1)=a<b=f(H_2),
\]
then, for every integer $c$ with $a<c<b$, there exists
$H\in\mathcal{F}$ such that $f(H)=c$.

The study of interpolation properties in graph theory has its origins in
questions concerning spanning trees and their numbers of leaves. Schuster
proved the following classical result.

\begin{theorem}[Schuster~\cite{Schuster}]
\label{thm:Schuster}
Let $G$ be a connected graph. If $G$ contains spanning trees with exactly
$a$ and $b$ leaves, where $a<b$, then, for every integer $c$ with
$a<c<b$, $G$ contains a spanning tree with exactly $c$ leaves.
\end{theorem}

Further interpolation properties for various graph invariants and families
of subgraphs have been extensively studied in
\cite{FisherEtAl,HararyHedetniemiPrins,HararyPlantholt,LihLinTong,
ToppVestergaard}.

Another well-known interpolation problem concerns cycle lengths. A graph
$G$ of order $n$ is said to be \emph{pancyclic} if it contains a cycle
$C_\ell$ for every integer $\ell$ with $3\le \ell\le n$. Bondy established the
following fundamental result on pancyclic graphs.

\begin{theorem}[Bondy~\cite{Bondy}]
\label{thm:Bondy}
Let $G$ be a Hamiltonian graph of order $n$ and size $m$. If
\[
    m\ge \frac{n^2}{4},
\]
then $G$ is pancyclic, unless $n$ is even and
$G\cong K_{n/2,n/2}$.
\end{theorem}

Since every cycle is a $2$-connected graph, it is natural to replace the
requirement of a cycle of each length by that of a $2$-connected subgraph
of each order. Under this weaker requirement, one may expect the
minimum-degree condition for pancyclicity to be relaxed. Motivated by this
observation, Yin and Wu~\cite{YinWu} investigated minimum-degree conditions
that guarantee the existence of $2$-connected subgraphs of all possible
orders. They proved the following result.

\begin{theorem}[Yin and Wu~\cite{YinWu}]
\label{thm:YinWu}
Let $G$ be a $2$-connected graph of order $n\ge 4$. If
\[
    \delta(G)\ge \left\lceil\frac{n}{3}\right\rceil+1,
\]
then $G$ contains a $2$-connected subgraph of order $\ell$ for every
$\ell\in\{4,5,\ldots,n\}$.
\end{theorem}

Liu and Ning~\cite{LiuNing} subsequently improved the minimum-degree
condition in Theorem~\ref{thm:YinWu} as follows.

\begin{theorem}[Liu and Ning~\cite{LiuNing}]
\label{thm:LiuNing}
Let $G$ be a $2$-connected graph of order $n\ge 4$. If
\[
    \delta(G)\ge \left\lceil\frac{n}{4}\right\rceil+2,
\]
then $G$ contains a $2$-connected subgraph of order $\ell$ for every
$\ell\in\{4,5,\ldots,n\}$.
\end{theorem}

They further proposed the following conjecture.

\begin{conj}[Liu and Ning~\cite{LiuNing}]
\label{conj:LiuNing}
For every fixed integer $q\ge 3$, there exists an integer $n_0(q)$ such
that every $2$-connected graph $G$ of order $n\ge n_0(q)$ with
\[
    \delta(G)\ge \frac{n}{q}
\]
contains a $2$-connected subgraph of order $\ell$ for every
$\ell\in\{4,5,\ldots,n\}$.
\end{conj}

In fact, we prove the following stronger result. Recall that a graph is
\emph{$k$-connected} if it has more than $k$ vertices and remains
connected after the deletion of any set of fewer than $k$ vertices. The
case $k=2$ of the theorem below confirms
Conjecture~\ref{conj:LiuNing}.

\begin{theorem}
\label{thm:main}
For every fixed pair of integers $k\ge 2$ and $q\ge 3$, there exists an
integer $n_0(k,q)$ such that every $k$-connected graph $G$ of order
$n\ge n_0(k,q)$ with
\[
    \delta(G)\ge \frac{n}{q}
\]
contains a $k$-connected subgraph of every order
\[
    \ell\in\{2k,2k+1,\ldots,n\}.
\]
\end{theorem}

The lower endpoint in Theorem~\ref{thm:main} cannot be improved in
general. Indeed, let $m\ge k$ and consider $K_{m,m}$. This graph is
$k$-connected and
\[
    \delta(K_{m,m})=m=\frac{|V(K_{m,m})|}{2}
    \ge \frac{|V(K_{m,m})|}{q}.
\]
On the other hand, every $k$-connected graph has minimum degree at
least $k$. Consequently, a $k$-connected subgraph of $K_{m,m}$ has at
least $k$ vertices in each part and therefore has order at least $2k$.

We now introduce some notation. Let $G$ be a graph. For a vertex
$v\in V(G)$, let $N_G(v)$ and $d_G(v)$ denote the neighborhood and degree
of $v$ in $G$, respectively. When there is no ambiguity, we write $N(v)$
and $d(v)$ instead. The minimum degree of $G$ is denoted by $\delta(G)$.
For a set $X\subseteq V(G)$, let $G[X]$ denote the subgraph of $G$
induced by $X$, and write $G-X$ for $G[V(G)\setminus X]$. For distinct
vertices $x,y\in V(G)$, an $x$--$y$ path is a path in $G$ with endpoints
$x$ and $y$, and its length is its number of edges. For a set $X$ and a nonnegative integer
$r$, let $\binomset{X}{r}$ denote the family of all $r$-element subsets
of $X$. The complete bipartite graph with partite sets of orders $s$ and
$t$ is denoted by $K_{s,t}$. All logarithms are
natural.

In Section~2, we establish several auxiliary lemmas and prove
Theorem~\ref{thm:main}. Concluding remarks are given in Section~3.

\section{\normalsize Proof of Theorem~\ref{thm:main}}

We first record a standard vertex-addition observation.

\begin{lemma}
\label{lem:addvertex}
Let $H$ be a $k$-connected graph, and let $H'$ be obtained from $H$
by adding a new vertex $v$ adjacent to at least $k$ vertices of $H$.
Then $H'$ is $k$-connected.
\end{lemma}

\begin{proof}
Let $X\subseteq V(H')$ with $|X|\le k-1$. If $v\in X$, then
\[
    H'-X=H-(X\setminus\{v\}),
\]
which is connected because $H$ is $k$-connected. If $v\notin X$, then
$H-X$ is connected. Moreover, at least one of the $k$ neighbors of
$v$ in $H$ lies outside $X$, so $v$ is joined to $H-X$ in $H'-X$.
Thus $H'-X$ is connected in either case. Since
$|V(H')|>|V(H)|>k$ and $X$ was arbitrary, the graph $H'$ is
$k$-connected.
\end{proof}

We shall use the following bounded-length version of Menger's theorem. For
positive integers $d$ and $m$, let $\cP_{d,m}(G)$ denote the property
that every pair of vertices of $G$ is joined by $m$ internally
vertex-disjoint paths, each of length at most $d$.

\begin{theorem}[Faudree et al.~\cite{FaudreeEtAl}]
\label{thm:shortmenger}
Let $G$ be an $m$-connected graph of order $n$. If
\[
    \delta(G)\ge
    \left\lfloor
      \frac{n-m+2}{\left\lfloor(d+4)/3\right\rfloor}
    \right\rfloor+m-2,
\]
then $G$ satisfies $\cP_{d,m}(G)$.
\end{theorem}

We shall also use the following two standard probability inequalities;
see Alon and Spencer~\cite{AlonSpencer}.

\paragraph*{Markov's inequality.}
If $X$ is a nonnegative random variable and $a>0$, then
\begin{equation}
\label{eq:markov}
    \Pr(X\ge a)\le \frac{\mathbb{E}X}{a}.
\end{equation}

\paragraph*{Chernoff's inequality.}
If $X$ is a sum of independent Bernoulli random variables with mean
$\mu$, then
\begin{equation}
\label{eq:chernoff}
    \Pr\bigl(X\le(1-\varepsilon)\mu\bigr)
    \le \exp\left(-\frac{\varepsilon^2\mu}{2}\right)
    \qquad (0<\varepsilon<1).
\end{equation}

We next construct a small set having many neighbors at every vertex.

\begin{lemma}
\label{lem:smallset}
For every fixed pair of integers $k\ge2$ and $q\ge3$, every sufficiently
large graph $G$ of order $n$ with $\delta(G)\ge n/q$ has a set
$S\subseteq V(G)$ such that
\[
    2k\le |S|\le 32q\log n+2k
\]
and
\[
    |N_G(v)\cap S|\ge k
    \qquad\text{for every }v\in V(G).
\]
\end{lemma}

\begin{proof}
Choose each vertex independently with probability
\[
    p=\frac{16q\log n}{n},
\]
where $n$ is sufficiently large that $p\le1$. Let $S_0$ be the
resulting random set. For a fixed vertex $v$, put
\[
    X_v=|N_G(v)\cap S_0|.
\]
Then
\[
    \mu_v:=\mathbb{E}X_v=pd_G(v)\ge16\log n.
\]
Since $k$ is fixed, by taking $n$ sufficiently large we may assume that
$k\le\mu_v/2$. By \eqref{eq:chernoff},
\[
    \Pr(X_v<k)
    \le \Pr\left(X_v\le\frac{\mu_v}{2}\right)
    \le \exp\left(-\frac{\mu_v}{8}\right)
    \le n^{-2}.
\]
The union bound therefore gives
\[
\Pr\bigl(\exists v\in V(G):X_v<k\bigr)
=
\Pr\left(\bigcup_{v\in V(G)}\{X_v<k\}\right)
\le
\sum_{v\in V(G)}\Pr(X_v<k)
\le n\cdot n^{-2}
=\frac1n.
\]
Moreover,
\[
    \mathbb{E}|S_0|=16q\log n,
\]
so \eqref{eq:markov} gives
\[
    \Pr\bigl(|S_0|>32q\log n\bigr)\le\frac12.
\]
Combining these estimates, we obtain
\[
\begin{aligned}
&\Pr\bigl(X_v\ge k\text{ for every }v\in V(G)
          \text{ and }|S_0|\le32q\log n\bigr)\\
&\qquad\ge
1-\Pr\bigl(\exists v\in V(G):X_v<k\bigr)
-\Pr\bigl(|S_0|>32q\log n\bigr)\\
&\qquad\ge 1-\frac1n-\frac12>0.
\end{aligned}
\]
Hence there exists a choice of $S_0$ satisfying both properties. Since $n$ is sufficiently large, we may assume that $n\ge2k$. If $|S_0|\ge2k$, let $S=S_0$; otherwise, add arbitrary vertices to
$S_0$ until the resulting set $S$ has size $2k$. Then
\[
    2k\le |S|\le32q\log n+2k,
\]
and, since $S_0\subseteq S$, every vertex has at least $k$ neighbors
in $S$.
\end{proof}

We now combine Lemma~\ref{lem:smallset} with
Theorem~\ref{thm:shortmenger}.

\begin{lemma}
\label{lem:core}
For every fixed pair of integers $k\ge2$ and $q\ge3$, every sufficiently
large $k$-connected graph $G$ of order $n$ with
$\delta(G)\ge n/q$ contains an induced $k$-connected subgraph $D$
satisfying
\[
    |V(D)|\le (32q\log n+2k)\bigl(1+k^2(6q-5)\bigr)
\]
and
\[
    |N_G(v)\cap V(D)|\ge k
    \qquad\text{for every }v\in V(G)\setminus V(D).
\]
\end{lemma}

\begin{proof}
Let $S$ be the set supplied by Lemma~\ref{lem:smallset}, and write
$s=|S|$. Set
\[
    d=6q-4.
\]
Then
\[
    \left\lfloor\frac{d+4}{3}\right\rfloor=2q.
\]
Since $k$ and $q$ are fixed, for all sufficiently large $n$ we have
\[
    \frac{n}{q}
    \ge
    \left\lfloor\frac{n-k+2}{2q}\right\rfloor+k-2.
\]
Thus Theorem~\ref{thm:shortmenger} implies that every pair of vertices
of $G$ is joined by $k$ internally vertex-disjoint paths, each of length
at most $d$.

Partition $S$ as $S=A\cup B$, where $|A|=k$ and
$|B|=s-k\ge k$. For every pair $x\in A$ and $y\in B$, choose $k$
internally vertex-disjoint $x$--$y$ paths in $G$, each of length at
most $d$. Let $U$ be the union of $S$ and the vertex sets of all these
paths, and define
\[
    D:=G[U].
\]
There are $|A||B|=k(s-k)\le ks$ such pairs. Since each chosen path
contributes at most $d-1$ vertices outside $S$, we have
\begin{align*}
    |V(D)|
    &\le s+k(d-1)|A||B|\\
    &\le s+k^2(d-1)s\\
    &=s\bigl(1+k^2(d-1)\bigr)\\
    &\le(32q\log n+2k)\bigl(1+k^2(6q-5)\bigr).
\end{align*}

It remains to show that $D$ is $k$-connected. Let
$X\subseteq V(D)$ with $|X|\le k-1$. We first show that all vertices of $S\setminus X$ lie in the same
component of $D-X$, and then show that every remaining vertex of
$D-X$ also lies in this component.

Since $|A|=k$ and $|B|\ge k$, both $A\setminus X$ and
$B\setminus X$ are nonempty. Choose
\[
    a_0\in A\setminus X
    \qquad\text{and}\qquad
    b_0\in B\setminus X.
\]
For each $x\in A\setminus X$, consider the $k$ chosen
$x$--$b_0$ paths. Since their endpoints avoid $X$ and their internal vertices are
pairwise disjoint, at most $|X|\le k-1$ of the $k$ paths meet $X$.
Hence at least one of them is contained in $D-X$. Therefore, every
vertex of $A\setminus X$ is joined to $b_0$ by a path in $D-X$.
Similarly, every vertex of $B\setminus X$ is joined to $a_0$ by a path
in $D-X$. Since $a_0\in A\setminus X$, the graph $D-X$ also contains an
$a_0$--$b_0$ path. Hence all vertices of $S\setminus X$ lie in the
same component of $D-X$.

Now consider any vertex
$v\in V(D)\setminus(S\cup X)$. By the choice of $S$, the vertex $v$
has at least $k$ neighbors in $S$. Since $|X|\le k-1$, at least one of these neighbors belongs to
$S\setminus X$. The corresponding edge belongs to $D-X$ because $D$
is induced and neither endpoint lies in $X$. Hence $v$ lies in the
component containing $S\setminus X$. It follows that $D-X$ is
connected. Since
\[
    |V(D)|\ge |S|\ge2k>k,
\]
the graph $D$ is $k$-connected.

Finally, every vertex of $V(G)\setminus V(D)$ has at least $k$
neighbors in $S$, and $S\subseteq V(D)$. Thus
\[
    |N_G(v)\cap V(D)|\ge k
    \qquad\text{for every }v\in V(G)\setminus V(D),
\]
as required.
\end{proof}

Now we are ready to prove Theorem~\ref{thm:main}.

\begin{proof}[\textup{\textbf{Proof of Theorem~\ref{thm:main}}}]
Fix integers $k\ge2$ and $q\ge3$. Choose $n_0(k,q)$ sufficiently large
that all preceding lemmas apply and, for every $n\ge n_0(k,q)$,
\begin{equation}
\label{eq:overlap}
    k+\frac{n}{(2q)^k}
    \ge
    (32q\log n+2k)\bigl(1+k^2(6q-5)\bigr),
\end{equation}
and $n\ge2q(k-1)$. Such a choice is possible because $k$ and $q$ are
fixed and $n/\log n\to\infty$.

Let $G$ be a $k$-connected graph of order $n\ge n_0(k,q)$ with
$\delta(G)\ge n/q$. By Lemma~\ref{lem:core}, there is an induced
$k$-connected subgraph $D$ such that
\begin{equation}
\label{eq:D-bound}
    |V(D)|
    \le
    (32q\log n+2k)\bigl(1+k^2(6q-5)\bigr),
\end{equation}
and every vertex outside $D$ has at least $k$ neighbors in $D$.

\paragraph*{Large orders:}
List the vertices outside $D$ as
\[
V(G)\setminus V(D)=\{x_1,x_2,\ldots,x_t\}.
\]
Let $D_0=D$. For $1\le i\le t$, let
\[
    D_i=G\bigl[V(D)\cup\{x_1,\ldots,x_i\}\bigr].
\]
Since every $x_i$ has at least $k$ neighbors in
$D\subseteq D_{i-1}$, Lemma~\ref{lem:addvertex} implies inductively that
every $D_i$ is $k$-connected. Consequently, $G$ contains a
$k$-connected subgraph of every order in
\begin{equation}
\label{eq:large-orders}
    \{|V(D)|,|V(D)|+1,\ldots,n\}.
\end{equation}

\paragraph*{Small orders:}
For each $k$-element subset $Q$ of $V(G)$, define
\[
    c(Q)=\left|\bigcap_{x\in Q}N_G(x)\right|.
\]
Double-counting the pairs $(v,Q)$ with $v\in V(G)$ and
$Q\in\binomset{N_G(v)}{k}$ gives
\[
    \sum_{Q\in\binomset{V(G)}{k}}c(Q)
    =\sum_{v\in V(G)}\binom{d_G(v)}{k}
    \ge n\binom{\delta(G)}{k}.
\]
Hence there exists $A\in\binomset{V(G)}{k}$ such that
\begin{align*}
    c(A)
    &\ge \frac{n\binom{\delta(G)}{k}}{\binom{n}{k}}\\
    &=n\prod_{i=0}^{k-1}\frac{\delta(G)-i}{n-i}.
\end{align*}
For every $0\le i\le k-1$, the assumptions $\delta(G)\ge n/q$ and
$n\ge2q(k-1)$ give
\[
    \frac{\delta(G)-i}{n-i}
    \ge \frac{n/q-(k-1)}{n}
    \ge \frac{1}{2q}.
\]
Therefore,
\begin{equation}
\label{eq:common-neighbors}
    c(A)
    \ge n\prod_{i=0}^{k-1}\frac{\delta(G)-i}{n-i}
    \ge n\left(\frac{1}{2q}\right)^k
    =\frac{n}{(2q)^k}.
\end{equation}
Notice that $\bigcap_{a\in A}N_G(a)$ is disjoint from $A$.
Since $n$ is sufficiently large, \eqref{eq:common-neighbors} implies
that $c(A)\ge k$. For every integer $\ell$ with
\[
    2k\le\ell\le k+c(A),
\]
we may therefore choose a set
\[
    B_\ell\subseteq \bigcap_{a\in A}N_G(a)
    \qquad\text{with}\qquad
    |B_\ell|=\ell-k.
\]
The graph $G[A\cup B_\ell]$ contains $K_{k,\ell-k}$ as a spanning
subgraph. Since $\ell-k\ge k$, the graph $K_{k,\ell-k}$ is
$k$-connected, and adding edges preserves $k$-connectivity. Thus $G$
contains a $k$-connected subgraph of every order in
\begin{equation}
\label{eq:small-orders}
    \{2k,2k+1,\ldots,k+c(A)\}.
\end{equation}

By \eqref{eq:common-neighbors}, \eqref{eq:overlap}, and
\eqref{eq:D-bound},
\[
    k+c(A)
    \ge k+\frac{n}{(2q)^k}
    \ge |V(D)|.
\]
Consequently, the ranges in \eqref{eq:large-orders} and
\eqref{eq:small-orders} overlap, so $G$ contains a $k$-connected
subgraph of every order
\[
    \ell\in\{2k,2k+1,\ldots,n\}.
\]
This completes the proof.
\end{proof}

\section{\normalsize Concluding remarks}

Theorem~\ref{thm:main} confirms the conjecture of Liu and Ning in the
case $k=2$ and establishes its higher-connectivity analogue for every
fixed $k\ge3$. As noted in the Introduction, the family of complete
bipartite graphs $K_{m,m}$ shows that the lower endpoint $2k$ is best
possible.

More generally, the same argument applies under the condition
$\delta(G)\ge\alpha n$ for any fixed $\alpha>0$, with a corresponding
threshold $n_0=n_0(k,\alpha)$.

The proof is asymptotic, and no attempt has been made to optimize the
threshold $n_0(k,q)$. Determining sharper bounds on $n_0(k,q)$ would be an interesting
direction for future research.

\medskip
\noindent\textbf{Statement of AI Use.}
The key application of the probabilistic method was suggested by ChatGPT 5.5. ChatGPT was also used for language
refinement, grammatical revision, and improving the overall presentation
of the manuscript. The author verified all arguments and takes full
responsibility for the manuscript.

\end{document}